\documentclass[a4paper]{article}
\usepackage{amsmath,amsthm,amssymb}

\usepackage{color}
\usepackage[dvips]{graphicx} 

\usepackage[all]{xy}

\newtheorem{thm}{Theorem}[section]
\newtheorem{defn}[thm]{Definition}
\newtheorem{ex}[thm]{Example}
\newtheorem{prop}[thm]{Proposition}

\newtheorem{cor}[thm]{Corollary}
\newtheorem{lem}[thm]{Lemma}

\newtheorem{rem}[thm]{Remark}

\newcommand{\mf}[1]{{\mathfrak{#1}}}
\newcommand{\mr}[1]{{\mathrm{#1}}}

\newcommand{\bb}[1]{{\mathbb{#1}}}
\newcommand{\mca}[1]{{\mathcal{#1}}}

\newcommand{\Hom}{\mr{Hom}}

\newcommand{\Z}{\bb{Z}}
\newcommand{\C}{\bb{C}}

\newcommand{\OO}{\mca{O}}
\newcommand{\R}{\bb{R}}
\newcommand{\Q}{\bb{Q}}
\newcommand{\HH}{\bb{H}}

\usepackage{version}
\newenvironment{NB}{
\color{red}{\bf NB}. \footnotesize
}{}

\excludeversion{NB}
\excludeversion{NB2}

\title{On higher rank Donaldson-Thomas invariants}
\author{Kentaro Nagao\\
RIMS, Kyoto University\\
  Kyoto 606-8502, Japan
}

\begin{document}

\maketitle
\begin{abstract}
We study higher rank Donaldson-Thomas invariants of a Calabi-Yau $3$-fold using Joyce-Song's wall-crossing formula.
We construct quivers whose counting invariants coincide with the Donaldson-Thomas invariants. 
As a corollary, we prove the integrality and a certain symmetry for the higher rank invariants.
\end{abstract}

\setcounter{tocdepth}{2}
\tableofcontents

\section*{Introduction}
The aim of this paper is to study the higher rank Donaldson-Thomas (DT) invariants 
for a Calabi-Yau $3$-fold $X$ using Joyce-Song's wall-crossing formula.

The DT invariant, introduced as a holomorphic analogue of the Casson invariant in \cite{thomas-dt}, is a counting invariant of stable coherent sheaves on $X$.
It is conjectured in \cite{mnop} that the DT theory is equivalent to the Gromov-Witten theory after being normalized by the zero dimensional DT invariants.
It is shown in \cite{behrend-fantechi,levine_pandharipande,li_0dimDT} that the generating function of the zero dimensional DT invariants is given by the MacMahon function:
\[
M(-q)^{\chi(X)}:=\prod_{k\geq 1}(1-(-q)^k)^{-k\chi(X)}.
\]

Recently, generalized DT invariants and their wall-crossing formula has been developed in \cite{ks} and \cite{joyce-song}.
The original DT invariant is given as the weighted Euler characteristic of the moduli scheme (\cite{behrend-dt}), which is an integer by definition.
In the generalized DT theory, we face difficulties to define an invariant since the moduli is not a scheme but a stack.
The idea in \cite{ks} and \cite{joyce-song}, which had appeared in \cite{joyce-1,joyce-2,joyce-3,joyce-4} already, was to use the {\it motivic Hall algebra} to define an invariant. 
Although the invariant is a rational number a priori, it is expected to be an integer (\cite[Conjecture 6]{ks}, \cite[Conjecture 6.13]{joyce-song}). 

In \S \ref{sec_06},
we study the higher rank zero dimensional DT invariant $\Omega(r,n)$ for integers $r$ and $n$, which is a counting invariant of stable coherent sheaves $E$ on $X$ such that 
\[
\mr{ch}(E):=(\mr{ch}_3(E),\mr{ch}_2(E),\mr{ch}_1(E),\mr{ch}_0(E))
=(r,0,0,-n)
\in \bigoplus_{i=0}^3H^{2i}(X,\Z).
\]
The integer $r$ is called the {\it rank}. 
Note that the original DT invariant is the rank one invariant. 
Such invariants have been studied in \cite[\S 6.5]{ks} as BPS invariants for D$0$-D$6$ bound states.
As pointed out there, we can compute them by the wall-crossing formula (\cite[Theorem 7,8]{ks}, \cite[Equation (79)]{joyce-song}) once we are given the data of the rank one invariants.
In recent the papers \cite{toda_rank2} and \cite{stoppa_D0D6}, the authors studied lower rank invariants by analyzing the wall-crossing formula directly. 
For higher rank invariants, the wall-crossing formula is complicated and it seems difficult to extend their arguments.
The main idea in this paper comes from the following observation:
\begin{quote}
the DT type invariants for one stability condition are determined by the initial data (the DT type invariants for the other stability condition) and the coefficients in the wall-crossing formula.
\end{quote}
Actually, we find a quiver without relations whose DT type theory has the same initial data and the same coefficients in the wall-crossing formula as D$0$-D$6$ state counting.
Then, the D$0$-D$6$ invariants coincide with the DT type invariants for the quiver.
In particular, we get the integrality for the D$0$-D$6$ invariants since the DT type invariants for a quiver without relations are known to be an integer (\cite[Theorem 7.28]{joyce-song}).

\smallskip

Here we give a brief review of the D$0$-D$6$ invariants following \cite{toda_rank2}.
First, we take a heart $\mca{A}_X$ of a bounded t-structure of the category of D$0$-D$6$ bound states (\cite[\S 2.1]{toda_rank2}).
We denote by $\Gamma_X:=H^6(X,\Z)\oplus H^0(X,\Z)$ its numerical Grothendieck group and by $\langle-,-\rangle_X\colon \Gamma_X\times\Gamma_X\to \Z$ the Euler pairing.
There are two stability conditions $Z^\pm_X\colon \Gamma_X\to\C$ and the DT type invariants $\Omega^\pm_X(r,n)$ ($(r,n)\in \Gamma_X$) are defined for each stability condition.
The invariant $\Omega^-_X(r,n)$ is easy to compute.
Once we are given the data of $\Omega^-_X(r,n)$, we can compute $\Omega^+_X(r,n)$ by the wall-crossing formula. 
Note that the wall-crossing formula depends only on the data of $\Gamma_X$, $\langle-,-\rangle_X$ and $Z_X^\pm$, not on the category $\mca{A}_X$.

In \S \ref{subsec_quiver}, we define a quiver $Q=Q_{\chi,N}$ ($\chi=\chi(X)$) for a positive integer $N$.
We denote by $\Gamma_Q$ the Grothendieck group of the category $\mr{mod}Q$ of finite dimensional $Q$-modules and by $\langle-,-\rangle_Q\colon \Gamma_Q\times\Gamma_Q\to \Z$ the Euler pairing.
Then, 
\begin{itemize}
\item there is a group homomorphism $\pi\colon\Gamma_Q\to\Gamma_X$ such that 
\begin{equation}\label{eq_1}
\langle-,-\rangle_Q=\langle\pi(-),\pi(-)\rangle_X,
\end{equation} 
\item the maps 
\begin{equation}\label{eq_2}
Z_Q^\pm:=Z_X^\pm\circ \pi
\end{equation}
give stability conditions on $\mr{mod}Q$ and the DT type invariants $\Omega^\pm_Q(\beta)$ ($\beta\in \Gamma_Q$) are defined,
\item the invariants $\Omega^-_Q(\beta)$ are easy to compute and if $n\leq N$ then we have
\begin{equation}\label{eq_3}
\Omega^-_X(r,n)=\sum_{\beta\in \pi^{-1}(r,n)}\Omega^-_Q(\beta),
\end{equation}
\item we can compute $\Omega^+_Q(\beta)$ by the wall-crossing formula which depends only on the data of $\Gamma_Q$, $\langle-,-\rangle_Q$ and $Z_Q^\pm$.
\end{itemize}
Thanks to Equation \eqref{eq_1} and \eqref{eq_2}, $\Omega^\pm_X(r,n)$ and $\Omega^\pm_Q(\beta)$ satisfy the same wall-crossing formula. 
Substituting Equation \eqref{eq_3} for the wall-crossing formula, we get 
\begin{equation}\label{eq_4}
\Omega^+_X(r,n)=\sum_{\beta\in \pi^{-1}(r,n)}\Omega^+_Q(\beta)
\end{equation}
for $n\leq N$.
Hence, the integrality of $\Omega^+_X(r,n)$ follows from the integrality of $\Omega^+_Q(\beta)$, which is shown in \cite[Theorem 7.28]{joyce-song}.

Another application of Equation \eqref{eq_4} is the following symmetry:
\[
\Omega^+_X(r,n)=\Omega^+_X(n-r,n).
\]
This is a consequence of the reflection functor in the sense of \cite{BGP_reflection}.
In \cite[\S 1.5]{stoppa_D0D6}, the author proved the symmetry using the correspondence with GW invariants (\cite{GPS_tropical_vertex}). Yukinobu Toda provided another proof, which is more direct.

\smallskip

In \S \ref{sec_026}, we study D$0$-D$2$-D$6$ invariants\footnote{In recent paper \cite{CDP_rank2_ADHM}, the authors studied rank two D$0$-D$2$-D$6$ invariants for local curves using the ADHM description.} for small crepant resolutions by the same method.

\smallskip

All the integrality results in this paper follow from more general result: 
if the integrality for one generic stability condition is true, 
then so is the one for another generic stability condition (``{\it relative integrality}''). 
The relative integrality for Kontsevich-Soibelman's wall-crossing formula was proved by Markus Reineke (\cite{reineke_integrality}). 
Martijn Kool proved the relative integrality for Joyce-Song invariants (\cite{joyce_kool})\footnote{His argument is quite similar to ours: he uses the wall-crossing formula and the integrality for quivers without relations. He proved the relative integrality (and absolute integrality for rank two invariants) in the summer of 2009 (with Dominic Joyce), long before the idea of this paper occurred to the author!}.

\subsection*{Acknowledgement}
The author is grateful to Martijn Kool for explaining his result on the relative integrality. 
He also thanks Duiliu-Emanuel Diaconescu, Dominic Joyce, Jacopo Stoppa and Yukinobu Toda for helpful comments.

This paper was written while the author has been visiting the University of Oxford.
He is grateful to Dominic Joyce for the invitation and to the Mathematical Institute for hospitality.

The author is supported by JSPS Fellowships for Young Scientists (No.\ 19-2672).

\begin{NB}
Although we will use Joyce-Song's formulation in the body of the paper, in the introduction we give an explanation in terms of Kontsevich-Soibelman's formulation, which will help the readers to grasp the outline of the argument.

First, we consider the two dimensional torus 
\[
\mr{QT}^{\mr{D}0\text{-}\mr{D}6}_{\mr{sc}}:=\C[x^\pm,y^\pm]
\]
and its automorphism $T_{r,n}$ given by 
\[
T_{r,n}(x)=(1-x^ry^n)^{-n}x,\quad 
T_{r,n}(y)=(1-x^ry^r)^{n}y.
\]
Due to the factorization property, the invariants $\Omega(r,n)$ are characterized by 
\[
\prod_{n\geq 1}T_{n,0}^{-\chi}\cdot T_{0,1}\cdot \prod_{n\geq 1}T_{n,0}^{\chi}
=
\prod^{\longrightarrow}_{n\geq 0,r\geq 1}T_{r,n}^{\Omega(r,n)}
\]
where $\chi:=\chi(X)$ (\cite[Equation (1.3)]{stoppa_D0D6}).

In \S \ref{}, we introduce the quiver $Q_{\chi,N}$ for a positive integer $N$.
Let 
\[
\mr{QT}^{Q_{\chi,N}}_{\mr{sc}}:=
\C[X^\pm,Y_{j,k}^\pm (1\leq j\leq \chi,1\leq k\leq r)]
\]
be the associated torus.
\end{NB}

\section{D$0$-D$6$ states}\label{sec_06}
\subsection{Quivers associated to D$0$-D$6$ states}
\subsubsection{}\label{subsec_quiver}
Let $\chi$ be an integer and $N$ be a positive integer. 
We put $I=I_{\chi,N}:=\{(j,k)\mid 1\leq j\leq |\chi|,1\leq k\leq N\}$ and $\bar{I}=\bar{I}_{\chi,N}:=\{0\}\sqcup I$.
Let $Q=Q_{\chi,N}$ be the quiver whose set of vertices is $\bar{I}$ and whose set of arrows is 
\[
\{a_{j,k,p}\mid (j,k)\in I,1\leq p\leq k\}
\]
if $\chi <0$ and 
\[
\{a_{j,k,p}\mid (j,k)\in I,1\leq p\leq k\}
\sqcup
\{b_{j,k}\mid (j,k)\in I\}
\]
if $\chi >0$ where $a_{j,k,p}$ is an arrow from the vertex $0$ to the vertex $(j,k)$ and $b_{j,k}$ is a loop from the vertex $(j,k)$ to itself.
\begin{ex}
The quiver for $(\chi,N)=(2,3)$ and $=(-2,3)$ is shown in Figure \ref{fig_1} and Figure \ref{fig_1.5} respectively.
\end{ex}
\begin{figure}[htbp]
  \centering
  \input{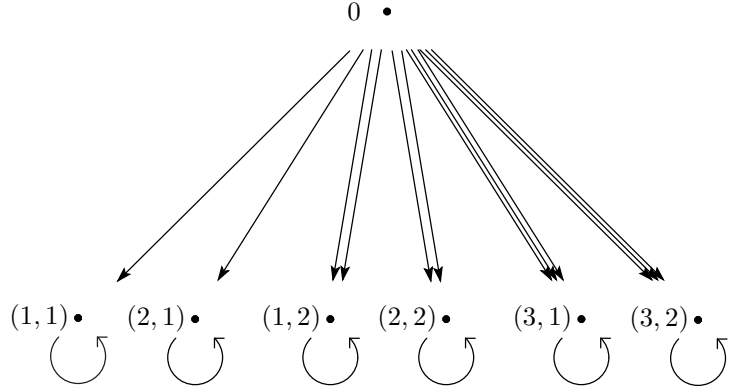}
  \caption{The quiver $Q_{2,3}$}
\label{fig_1}
\end{figure}
\begin{figure}[htbp]
  \centering
  \input{quiver2.5.tpc}
  \caption{The quiver $Q_{-2,3}$}
\label{fig_1.5}
\end{figure}

\subsubsection{}
Let ${\mr{mod}Q}$ be the category of finite dimensional $Q$-modules and 
\[
\Gamma_Q:=K({\mr{mod}(Q)})\simeq \Z^{\bar{I}}=\Z\oplus \Z^I
\]
be the Grothendieck group. 
Let $s_0$ and $s_{j,k}$ be the simple modules for the vertices and we put
\[
e_0:=[s_0],\quad e_{j,k}=[s_{j,k}]\in \Gamma_Q.
\]

We denote by $\HH\subset\C$ the upper half plane,
\[
\HH:=\{z\in\C\mid \mr{Im}(z)>0\}.
\]
We fix four complex numbers $\theta^\pm_r$, $\theta^\pm_n\in\HH$ such that
\[
\mr{arg}\theta^+_r>\mr{arg}\theta^+_n,
\quad
\mr{arg}\theta^-_r<\mr{arg}\theta^-_n.
\]
Let 
\[
Z_Q^\pm\colon\Gamma_Q\to \C
\]
be the stability conditions on $\mr{mod}Q$ given by 
\[
Z_Q^\pm(e_0)=\theta^\pm_r,\quad 
Z_Q^\pm(e_{j,k})=k\cdot \theta^\pm_n\ (\forall (j,k)\in I).
\]
It is easy to classify $Z_Q^-$-semistable modules:
\begin{lem}\label{lem_12}
A $Q$-module $V$ is $Z_Q^-$-semistable if and only if 
$V_0=0$ or $V$ is supported on the vertex $0$.
\end{lem}

\subsubsection{}
For $\beta\in \Gamma_Q$, let $\mca{O}bj_Q^\beta$ be the moduli stack of all $Q$-modules with dimension vectors $\beta$ and\[
\mca{M}_Q^\beta(Z_Q^\pm)\subset \mca{O}bj_Q^\beta
\]
be the substack of $Z_Q^\pm$-semistable modules.

Let $(\mca{H}(Q),*)$ be the Hall algebra associated to the Abelian category of $Q$-modules (\cite{joyce-2})\footnote{The algebra $\mca{H}(Q)$ is denoted by $\underline{\mr{SF}}(\mca{O}bj_Q)$ in Joyce's paper. An element in $\underline{\mr{SF}}(\mf{Obj}_Q)$ is called a {\it stack function} on $\mca{O}bj_Q$}. 
We define elements
\[
\delta_Q^{\beta}(Z_Q^\pm):=\bigl[\mca{M}_Q^{\beta}(Z_Q^\pm)\subset \mca{O}bj_Q^{\beta}\bigr]\in \mca{H}(Q)
\]
and
\[
\varepsilon_Q^{\beta}(Z_Q^\pm):=
\sum_{
\begin{subarray}{c}
l \geq 0,\ \beta_1+\cdot+\beta_l=\beta,\\
Z^-_Q(\beta_i)\in \R_{>0}\cdots Z^-_Q(\beta) \text{ for all $i$}.
\end{subarray}
}
\frac{(-1)^{l-1}}{l}\delta_Q^{\beta_1}(Z_Q^\pm)*\cdots*\delta_Q^{\beta_l}(Z_Q^\pm)
\]
The important fact \cite[Theorem 8.7]{joyce-3} is that $\varepsilon_Q^{\beta}(Z_Q^\pm)$ is supported on ``virtual indecomposable objects'', and we can define the weighted Euler characteristic
\[
\bar{\mr{DT}}_Q^\pm(\beta):=\chi(\varepsilon_Q^{\beta}(Z_Q^\pm),-\nu)\in\Q
\]
where $\nu$ is the Behrend function. 
We define the BPS invariant $\Omega_Q^\pm(\beta)\in \Q$ by the following equation:
\[
\bar{\mr{DT}}_Q^\pm(\beta)=
\sum_{m\geq 1,m|\beta}\Omega_Q^\pm(\beta/m)/m^2.
\]
Since the stability conditions $Z^\pm_Q$ are {\it generic} in the sense of \cite[Theorem 7.28]{joyce-song}, we can apply the integrality theorem:
\begin{thm}[\protect{\cite[Theorem 7.28]{joyce-song}}]\label{thm_js_integrality}
$\Omega_Q^\pm(\beta)\in \Z$.
\end{thm}
\begin{lem}\label{lem_Q-}
\[
\Omega_Q^-(\beta)=
\begin{cases}
1, & \beta=e_0,\\
-\mr{sgn}(\chi), & \beta=n\cdot e_{j,k}\ (n>0),\\
0, & \text{otherwise}.
\end{cases}
\]
\end{lem}
\begin{proof}
The equation follows from Lemma \ref{lem_12}, \cite[Equation (90)]{joyce-song} and Lemma \ref{lem_app}.
\begin{NB}
Let $Q'$ be the quiver given from $Q$ by removing the vertex $0$ and arrows $\{a_{j,k,p}\}$.
We want to compute the DT type invariants for $Q'$. 
Since the Euler pairing for $Q'$ vanishes, the invariants do not depend on stability conditions. 
Take a stability condition such that each simple module has distinct phase, 
then any semistable module is supported on a single vertex.
So the second equation follows from Lemma \ref{lem_12}, \cite[Equation (90)]{joyce-song} and Lemma \ref{lem_app}.
The first and third equations follow directly from the definition and Lemma \ref{lem_12}.
\end{NB} 
\end{proof}

\subsubsection{}
For a nonnegative integer $l$ and a sequence $\vec{\beta}=(\beta_1,\ldots,\beta_l)\in (\Gamma_Q)^l$, the rational number
\[
U(\vec{\beta};Z_Q^-,Z_Q^+)\in \Q
\]
is given by Equation (36) in \cite{joyce-song}.
The following is the Joyce-Song's wall-crossing formula (\cite[Equation (79)]{joyce-song}):
\begin{align}
\notag\bar{\mr{DT}}_Q^+(\beta)
&=
\sum_{
\begin{subarray}{c}
l\geq 1,\vec{\beta}\in (\Gamma_Q)^l,\\
\beta_1+\cdots+\beta_l=\beta.
\end{subarray}
}
\ \sum_{
\begin{subarray}{c}
\text{connected simply-connected}\\
\text{oriented graphs $\Upsilon$ with vertices $1,\ldots,l$,}\\
\text{$\underset{i}{\bullet}\to\underset{j}{\bullet}$ implies $i<j$.}
\end{subarray}
}\\
&\biggl(-\frac{1}{2}\biggr)^{l-1}U(\vec{\beta};Z_Q^-,Z_Q^+)
\sum_{\underset{i}{\bullet}\to\underset{j}{\bullet}\text{ in $\Upsilon$}}(-1)^{\langle\beta_i,\beta_j\rangle}\langle\beta_i,\beta_j\rangle\prod_i\bar{\mr{DT}}_Q^-(\beta_i).\label{eq_5}
\end{align}

\subsection{D$0$-D$6$ state counting}
Let $X$ be a smooth projective Calabi-Yau $3$-fold over $\C$, i.e.
\[
K_X\simeq \OO_X,\quad H^1(X,\OO_X)=0.
\]
Let $\mr{Coh}_0(X)$ be the category of coherent sheaves on $X$ with $0$-dimensional supports.
We denote by $\mca{A}_X$ the Abelian category of triples
\[
(\OO_X^{\oplus r},F,s)
\]
where $r$ is a nonnegative integer, $F\in \mr{Coh}_0(X)$ and $s\colon \OO_X^{\oplus r}\to F$. 
We set $\Gamma_X:=\Z\oplus\Z$ and a group homomorphism
\[
\mr{cl}\colon K(\mca{A}_X)\to \Gamma_X
\]
by 
\[
\mr{cl}(\OO_X^{\oplus r},F,s)
:=
(r,\mr{length}F).
\]
The Euler pairing $\langle-,-\rangle_X\colon \Gamma_X\times \Gamma_X\to \Z$ is given by 
\[
\langle(r,n),(r',n')\rangle_X=rn'-r'n.
\]
Let 
\[
Z_X^\pm\colon\Gamma_X\to \C
\]
be the group homomorphisms given by
\[
Z_X^\pm(1,0)=\theta^\pm_r,\quad 
Z_X^\pm(0,1)=\theta^\pm_n.
\]
They give stability conditions on $\mca{A}_X$ and we can define the invariants 
$\bar{DT}_X^\pm(r,n)$ and $\Omega_X^\pm(r,n)$ in the same way (see \cite{toda_rank2} for the details).
\begin{lem}[\protect{see \cite[Remark 3.10]{toda_rank2}}]\label{lem_X-}
\[
\Omega_X^-(r,n)=
\begin{cases}
1, & (r,n)=(1,0),\\
-\chi(X), & r=0,\, n>0,\\
0, & \text{otherwise}.
\end{cases}
\]
\end{lem}

\subsection{Main theorem}
We define the group homomorphism
\[
\pi\colon \Gamma_Q\twoheadrightarrow \Gamma_X
\]
by 
\[
\pi(e_0):=(1,0),\quad \pi(e_{j,k}):=(0,k).
\]
\begin{lem}\label{lem_34}
\begin{enumerate}
\item[\textup{(1)}] $Z^\pm_Q=Z^\pm_X\circ\pi$.
\item[\textup{(2)}] $\langle-,-\rangle_Q=\langle\pi(-),\pi(-)\rangle_X$.
\item[\textup{(3)}] If $n\leq N$, then  we have
\[
\Omega_X^-(r,n)=\sum_{\beta\in\pi^{-1}(r,n)}\Omega_Q^-(\beta).
\]
\end{enumerate}
\end{lem}
\begin{proof}
The first and second claims follow directly from the definitions.
The third one is a consequence of Lemma \ref{lem_Q-} and Lemma \ref{lem_X-}.
\end{proof}
\begin{cor}\label{cor_32}
If $n\leq N$, then we have
\[
\bar{\mr{DT}}_X^-(r,n)=
\sum_{\beta\in\pi^{-1}(r,n)}\bar{\mr{DT}}_Q^-(\beta).
\]
\end{cor}
\begin{proof}
\begin{align*}
\bar{\mr{DT}}_X^-(r,n)&=
\sum_{m|(r,n)}\Omega^-_X(r/m,n/m)/m^2\\
&=\sum_{m|(r,n)}\sum_{\beta\in\pi^{-1}{(r/m,n/m)}}\Omega^-_Q(\beta)/m^2\\
&=\sum_{m,\beta;\pi(m\beta)=(r,n)}\Omega^-_Q(\beta)/m^2\\
&=\sum_{\beta'\in \pi^{-1}(r,n)}\sum_{m|\beta'}\Omega^-_Q(\beta'/m)/m^2\\
&=\sum_{\beta'\in \pi^{-1}(r,n)}\bar{\mr{DT}}_Q^-(\beta').
\end{align*}
\end{proof}

For a nonnegative integer $l$ and a sequence $\vec{\alpha}=(\alpha_1,\ldots,\alpha_l)\in (\Gamma_X)^l$, the rational number
\[
U(\vec{\alpha};Z_X^-,Z_X^+)\in \Q
\]
is given by Equation (36) in \cite{joyce-song}.
\begin{lem}\label{lem_33}
\[
U(\vec{\beta};Z_Q^-,Z_Q^+)
=
U(\pi(\vec{\beta});Z_X^-,Z_X^+).
\]
\end{lem}
\begin{proof}
The claim follows directly from the definition \cite[Equation (36)]{joyce-song} and Lemma \ref{lem_34} (1).
\end{proof}
The following is the main theorem of this paper:
\begin{thm}\label{thm_main}
For $\alpha=(r,n)\in\Gamma_X$ with $n\leq N$, we have
\[
\bar{\mr{DT}}_X^+(\alpha)=\sum_{\beta\in\pi^{-1}(\alpha)}\bar{\mr{DT}}_Q^+(\beta).
\]
\end{thm}
\begin{proof}
\begin{align*}
&\bar{\mr{DT}}_X^+(\alpha)
\overset{\text{\cite[(79)]{joyce-song}}}{=}
\sum_{
\begin{subarray}{c}
l\geq 1,\vec{\alpha}\in (\Gamma_X)^l,\\
\alpha_1+\cdots+\alpha_l=\alpha.
\end{subarray}
}
\ \sum_{
\begin{subarray}{c}
\text{connected simply-connected}\\
\text{oriented graphs $\Upsilon$ with vertices $1,\ldots,l$,}\\
\text{$\underset{i}{\bullet}\to\underset{j}{\bullet}$ implies $i<j$.}
\end{subarray}
}\\
&\quad\quad\quad\quad\biggl(-\frac{1}{2}\biggr)^{l-1}U(\vec{\alpha};Z_X^-,Z_X^+)
\sum_{\underset{i}{\bullet}\to\underset{j}{\bullet}\text{ in $\Upsilon$}}(-1)^{\langle\alpha_i,\alpha_j\rangle}\langle\alpha_i,\alpha_j\rangle\prod_i\bar{\mr{DT}}_X^-(\alpha_i)\\
&\overset{\text{Corollary \ref{cor_32}}}{=}
\sum_{
\begin{subarray}{c}
l\geq 1,\vec{\alpha}\in (\Gamma_X)^l,\\
\alpha_1+\cdots+\alpha_l=\alpha.
\end{subarray}
}
\ \sum_{
\begin{subarray}{c}
\text{connected simply-connected}\\
\text{oriented graphs $\Upsilon$ with vertices $1,\ldots,l$;}\\
\text{$\underset{i}{\bullet}\to\underset{j}{\bullet}$ implies $i<j$.}
\end{subarray}
}
\biggl(-\frac{1}{2}\biggr)^{l-1}U(\vec{\alpha};Z_X^-,Z_X^+)
\\
&
\quad\quad\quad\quad\sum_{\underset{i}{\bullet}\to\underset{j}{\bullet}\text{ in $\Upsilon$}}(-1)^{\langle\alpha_i,\alpha_j\rangle}\langle\alpha_i,\alpha_j\rangle\prod_i\,\Biggl(\,\sum_{\beta_i\in\pi^{-1}(\alpha_i)}\bar{\mr{DT}}_Q^-(\beta_i)\,\Biggr)\\
&\overset{\text{Lemma \ref{lem_34} (2)}}{\overset{\text{and Lemma \ref{lem_33}}}{=}}
\sum_{
\begin{subarray}{c}
l\geq 1,\vec{\alpha}\in (\Gamma_X)^l,\\
\alpha_1+\cdots+\alpha_l=\alpha.
\end{subarray}
}
\ \sum_{
\begin{subarray}{c}
\text{connected simply-connected}\\
\text{oriented graphs $\Upsilon$ with vertices $1,\ldots,l$;}\\
\text{$\underset{i}{\bullet}\to\underset{j}{\bullet}$ implies $i<j$.}
\end{subarray}
}
\ \sum_{
\begin{subarray}{c}
\vec{\beta}\in (\Gamma_Q)^l,\\
\beta_i\in\pi^{-1}(\alpha_i).
\end{subarray}
}
\\
&
\quad\quad\quad\quad\biggl(-\frac{1}{2}\biggr)^{l-1}U(\vec{\beta};Z_Q^-,Z_Q^+)
\sum_{\underset{i}{\bullet}\to\underset{j}{\bullet}\text{ in $\Upsilon$}}(-1)^{\langle\beta_i,\beta_j\rangle}\langle\beta_i,\beta_j\rangle\prod_i\bar{\mr{DT}}_Q^-(\beta_i)\\
&\quad\quad = \quad\quad
\sum_{\beta\in\pi^{-1}(\alpha)}\ 
\sum_{
\begin{subarray}{c}
l\geq 1,\vec{\beta}\in (\Gamma_Q)^l,\\
\beta_1+\cdots+\beta_l=\beta.
\end{subarray}
}
\ \sum_{
\begin{subarray}{c}
\text{connected simply-connected}\\
\text{oriented graphs $\Upsilon$ with vertices $1,\ldots,l$;}\\
\text{$\underset{i}{\bullet}\to\underset{j}{\bullet}$ implies $i<j$.}
\end{subarray}
}
\\
&
\quad\quad\quad\quad\biggl(-\frac{1}{2}\biggr)^{l-1}U(\vec{\beta};Z_Q^-,Z_Q^+)
\sum_{\underset{i}{\bullet}\to\underset{j}{\bullet}\text{ in $\Upsilon$}}(-1)^{\langle\beta_i,\beta_j\rangle}\langle\beta_i,\beta_j\rangle\prod_i\bar{\mr{DT}}_Q^-(\beta_i)\\
&\overset{\text{Equation \eqref{eq_5}}}{=}\sum_{\beta\in\pi^{-1}(\alpha)}\bar{\mr{DT}}_Q^+(\beta).
\end{align*}
\end{proof}
\begin{cor}\label{cor_main}
For $\alpha=(r,n)\in\Gamma_X$ with $n\leq N$, we have
\[
\Omega_X^+(\alpha)=\sum_{\beta\in\pi^{-1}(\alpha)}\Omega_Q^+(\beta).
\]
\end{cor}
\begin{proof}
Using Theorem \ref{thm_main}, we can show the claim in a similar way as Corollary \ref{cor_32}.
\end{proof}
\begin{thm}
$\Omega_X^+(r,n)\in\Z$.
\end{thm}
\begin{proof}
Take an integer $N\geq n$. Then the claim follows from Theorem \ref{thm_js_integrality} and Corollary \ref{cor_main}.
\end{proof}

\subsection{Symmetry}
Let $V$ be a $Z_Q^+$-semistable $Q$-module with dimension vector $(r,\mathbf{n})$ such that $\mathbf{n}\neq \vec{0}$ and $\pi(r,\mathbf{n})=(r,n)$.
We set 
\[
\tilde{V}:=\bigoplus_{(j,k,p)}V_{j,k},\quad f_V:=\bigoplus_{(j,k,p)}a_{j,k,p}\colon V\to \tilde{V}.
\]
Since $V$ is $Z_Q^+$-semistable, $f_V$ is injective. 
\begin{prop}\label{prop_dual}
For $\mathbf{n}\neq \vec{0}$, we have
\[
\Omega^+_Q(r,\mathbf{n})=\Omega^+_Q(n-r,\mathbf{n}).
\]
\end{prop}
\begin{proof}
Let $Q^{\mr{op}}$ be the opposite quiver of $Q$, that is, $Q^{\mr{op}}$ is the quiver with the same vertex set as $Q$ and given by reversing all arrows in $Q$. Let $Z^\pm_{Q^{\mr{op}}}$ be the stability conditions given from $Z^\pm_{Q}$ by the canonical identification $\Gamma_Q=\Gamma_{Q^{\mr{op}}}$. 

By taking vector space dual, we get the natural isomorphism between moduli stack of $Z_Q^+$ semistable $Q$-modules and the one of $Z_{Q^{\mr{op}}}^-$ semistable ${Q^{\mr{op}}}$-modules with the same dimension vectors.

By taking $\mr{coker}f_V$, we get the natural isomorphism between moduli stack of $Z_Q^+$ semistable $Q$-modules with dimension vectors $(r,\mathbf{n})$ and the one of $Z_{Q^{\mr{op}}}^-$ semistable ${Q^{\mr{op}}}$-modules with dimension vectors $(n-r,\mathbf{n})$.

Hence the claim follows.
\end{proof}
\begin{rem}
The operator taking $\mr{coker}f_V$ is nothing but the reflection functor in the sense of \cite{BGP_reflection}. 
A similar argument has appeared in \cite[\S 5.3]{gross_pandharipande}.
\end{rem}
\begin{thm}\label{thm_dual}
For $n\neq 0$, we have
\[
\Omega^+_X(r,n)=\Omega^+_X(n-r,n).
\]
\end{thm}
\begin{proof}
The claims follows from Corollary \ref{cor_main} and Proposition \ref{prop_dual}.
\end{proof}

\subsection{Appendix}
\subsubsection{MacMahon function via $\Omega_Q^+(1,\mathbf{n})$}
In this subsection, we give an explicit computation of the invariants $\Omega_Q^+(1,\mathbf{n})$ for $\chi=1$.

First, note that the moduli stack $\mca{M}^{(1,\mathbf{n})}(Z^+_Q)$ has the following expression:
\[
\mca{M}^{(1,\mathbf{n})}(Z^+_Q)
=\bigl[\mr{M}^{(1,\mathbf{n})}(Z^+_Q)/C^*\bigr],
\]
where $\mr{M}^{(1,\mathbf{n})}(Z^+_Q)$ is a smooth scheme of dimension 
\[
d(\mathbf{n}):=\sum_{(j,k)\in I}k\cdot \mathbf{n}_{j,k}.
\]
Hence we have 
\begin{equation}\label{eq_6}
\Omega^+_Q(1,\mathbf{n})=(-1)^{d(\mathbf{n})}\chi(\mr{M}^{(1,\mathbf{n})}(Z^+_Q)).
\end{equation}

We put $J:=\{(j,k,p)\mid (j,k)\in I, 1\leq p\leq k\}$.
For a map $f\colon J\to \Z_{\geq 0}$, we define a $Q$-module $V(f)$ as follows: 
\[
V(f)_0:=\C\cdot v_0,\quad V(f)_{j,k}:=\bigoplus_{1\leq p\leq k}\bigoplus_{1\leq a\leq f(j,k,p)}\C\cdot v_{j,k,p,a}
\]
and 
\begin{align*}
a_{j,k,p}(v_0)&:=
\begin{cases}
0, & f(j,k,p)=0,\\
v_{j,k,p,1}, & \text{otherwise},
\end{cases}\\
b_{j,k}(v_{j,k,p,a})&:=
\begin{cases}
0, & a=f(j,k,p),\\
v_{j,k,p,a+1}, & \text{otherwise}.
\end{cases}
\end{align*}

Let $A$ be the number of arrows in $Q$. 
The $A$-dimensional torus $T:=(\C^*)^A$ acts on $\mr{M}^{(1,\mathbf{n})}(Z^+_Q)$.
We can check that the $T$-fixed point set is isolated and
\begin{equation}\label{eq_55}
\mr{M}^{(1,\mathbf{n})}(Z^+_Q)^T=\Bigl\{V(f)\,\Big|\, \sum_p f(j,k,p)=\mathbf{n}_{j,k}\Bigr\}.
\end{equation}
Let $\mathbf{q}_{j,k}$ ($(j,k)\in I$) be formal variables and we set
\[
\mathbf{q}^{\mathbf{n}}:=\prod_{(j,k)\in I} (\mathbf{q}_{j,k})^{\mathbf{n}_{j,k}}.
\]
\begin{thm}
\[
\sum_{\mathbf{n}} \Omega^+_Q(1,\mathbf{n})\cdot \mathbf{q}^{\mathbf{n}}=\prod_{(j,k)\in I}(1-(-1)^k\mathbf{q}_{j,k})^{-1}
\]
\end{thm}
\begin{proof}
The claim follows from Equation \eqref{eq_6}, \eqref{eq_55} and 
\[
\chi\bigl(\mr{M}^{(1,\mathbf{n})}(Z^+_Q)\bigr)=\bigl|\mr{M}^{(1,\mathbf{n})}(Z^+_Q)^T\bigr|.
\]
\end{proof}
\begin{cor}
\[
\sum_{\mathbf{n}} \Omega^+_Q(1,\mathbf{n})\cdot \mathbf{q}^{\mathbf{n}}\big|_{\mathbf{q}_{j,k}=q^k}=M(-q)^{\chi}.
\]
\end{cor}
\begin{rem}
This is compatible with Corollary \ref{cor_main} and the formula for (rank one) degree zero DT invariants (\cite{behrend-fantechi,levine_pandharipande,li_0dimDT}):
\[
\sum_n \Omega^+_X(1,n)\cdot q^n = M(-q)^{\chi}.
\]
\end{rem}

\subsubsection{Counting invariants for the one loop quiver}
Let ${Q^\circ}$ be a quiver with a single vertex and a single loop. 
We want to compute the DT type invariants $\Omega_{Q^\circ}(n)$ for this quiver. 
First, we compute the pair invariants in the sense of \cite{joyce-song}.
Let $\mca{M}^{n,1}_{\mr{fr},{Q^\circ}}$ be the moduli scheme of pairs $(V,v)$, where $V$ is a $n$-dimensional ${Q^\circ}$-module and $v\in V$ is an element such that $\C {Q^\circ}\cdot v=V$. 
We define the pair invariant by
\[
\mr{NDT}^{n,1}_{Q^\circ}:=\chi(\mca{M}^{n,1}_{\mr{fr},{Q^\circ}},\nu)
\]
where $\nu$ is the Behrend function on $\mca{M}^{n,1}_{\mr{fr},{Q^\circ}}$.
The moduli scheme $\mca{M}^{n,1}_{\mr{fr},{Q^\circ}}$ coincides with the Hilbert scheme $\mr{Hilb}^n(\C)$ of $n$-points on $\C$ and we have
\[
\mr{Hilb}^n(\C)=\mr{Sym}^n(\C)\simeq \C^n.
\]
Hence we have $\mr{NDT}^{n,1}_{Q^\circ}=(-1)^n$.

Applying \cite[Corollary 7.23]{joyce-song}, we get
\[
1+\sum_{n>0}\mr{NDT}^{n,1}_{Q^\circ}\cdot q^n=\mr{exp}\Bigl(-\sum_{n> 0}(-1)^n\cdot n\cdot\bar{\mr{DT}}_{Q^\circ}(n)\cdot q^n\Bigr).
\]
Hence we have the following:
\begin{lem}\label{lem_app}
\[
\bar{\mr{DT}}_{Q^\circ}(n)=-1/n^2,
\quad\quad
\Omega_{Q^\circ}(n)=\begin{cases}
-1, & n=1,\\
0, & \text{otherwise}.
\end{cases}
\]
\end{lem}

\section{D$0$-D$2$-D$6$ states for small crepant resolutions}\label{sec_026}
\subsection{General statement}\label{subsec_gen}
The argument in \S \ref{sec_06} can be directly generalized as follows.

Let $\mca{A}$ be an Abelian category on which Joyce-Song's theory (including the definition of the invariants and the wall-crossing formula) works. 
Assume that there exists a stability condition $Z^\circ\colon \Gamma_{\mca{A}}\to \C$ which is generic in the sense of \cite[Theorem 7.28]{joyce-song} and satisfies the following conditions:
\begin{itemize}
\item $\Omega^{Z^\circ}_{\mca{A}}(\gamma)\in \Z$ for any $\gamma\in \Gamma$, 
\item if $\chi(\gamma,\gamma')>0$ for $\gamma,\gamma'\in\Gamma$ then $\mr{arg}{Z^\circ}(\gamma)<\mr{arg}{Z^\circ}(\gamma')$, and 
\item for any $\gamma\in\Gamma$, let $\Gamma_\gamma$ be the subset of $\Gamma$ whose element $\mu\in \Gamma_\gamma$ satisfies the following conditions:
\begin{itemize}
\item $\Omega^{Z^\circ}_{\mca{A}}(\mu)\neq 0$, 
\item there exist an integer $l\geq 0$ and a sequence $\mu_1,\ldots,\mu_l\in \Gamma$ such that $\mu+\sum\mu_i=\gamma$ and $\Omega^{Z^\circ}_{\mca{A}}(\mu_i)\neq 0$ ($1\leq i\leq l$).
\end{itemize}
Then $\Gamma_\gamma$ is a finite set.
\end{itemize}
For $\gamma\in\Gamma$, we set 
\[
I_\gamma:=
\{
(\mu,j)\mid \mu\in\Gamma_\gamma,\,1\leq j\leq |\Omega_{\mca{A}}^{Z^\circ}(\mu)|
\}.
\]
Let $Q=Q_\gamma$ be the quiver with its vertex set $I_\gamma$ and with
\begin{itemize}
\item $\chi(\mu,\mu')$ arrows from the vertex $(\mu,j)$ to the vertex $(\mu',j')$ if $\chi(\mu,\mu')>0$, and 
\item a loop from the vertex $(\mu,j)$ to itself if $\Omega_{\mca{A}}^{Z^\circ}(\mu)<0$.
\end{itemize}
We call $Q_\gamma$ as the {\it BPS quiver} of type $\gamma$ for $\mca{A}$. 
Let $\pi\colon \Gamma_{Q_\gamma}\to \Gamma_{\mca{A}}$ be the group homomorphism given by 
\[
\pi(e_{(\mu,j)})=\mu
\]
and let $Z_{\mca{A}}\colon \Gamma_{\mca{A}}\to \C$ be a generic stability parameter. We put $Z_{Q_\gamma}:=Z_{\mca{A}}\circ \pi$.
\begin{thm}
\[
\Omega_{\mca{A}}^{Z_{\mca{A}}}(\gamma)=\sum_{\lambda\in\pi^{-1}(\gamma)}\Omega^{Z_{Q_\gamma}}_{Q_\gamma}(\lambda).
\]
In particular, $\Omega^{Z_{\mca{A}}}_{\mca{A}}(\gamma)\in \Z$.
\end{thm}

\subsection{BPS quiver for the conifold}\label{subsec_BPS_conifold}
\subsubsection{}
Let $\hat{Q}=\hat{Q}_{\mr{conifold}}$ be the quiver in Figure \ref{pic2} and $\omega$ be the following superpotential: 
\[
\omega=a_1b_1a_2b_2-a_1b_2a_2b_1.
\]
\begin{figure}[htbp]
  \centering
  \input{pic2.tpc}
  \caption{The quiver $\hat{Q}_{\mr{conifold}}$}
\label{pic2}
\end{figure}
Let $\mathrm{mod}(\hat{Q},w)$ be the category of finite dimensional modules over the quiver with relations $\C\hat{Q}/\hat{I}_w$ where $\hat{I}_w$ is the two-sided ideal generated by the derivatives of $w$. 
For a stability condition $Z\colon \Gamma_{\hat{Q}}:=\Z^3\to\C$ on $\mathrm{mod}(\hat{Q},w)$ and a dimension vector $\mathbf{v}=(v_\infty,v_0,v_1)\in \Gamma_{\hat{Q}}$, we can define the DT type invariant $\Omega^Z_{\hat{Q},w}(\mathbf{v})$ (\cite[\S 7]{joyce-song}).

Take a stability condition $Z^\circ\colon \Gamma_{\hat{Q}}\to\C$ such that 
\[
\mr{arg}(Z^\circ(e_\infty))<\mr{arg}(Z^\circ(e_0))=\mr{arg}(Z^\circ(e_1)).
\]
\begin{lem}\label{lem_coni}
\[
\Omega^{Z^\circ}_{\hat{Q},w}(v_\infty,v_0,v_1)=
\begin{cases}
1, & (v_\infty,v_0,v_1)=(1,0,0),\\
-2, & (v_\infty,v_0,v_1)=(0,n,n)\ (n> 0),\\
1, & (v_\infty,v_0,v_1)=(0,n,n+1)\text{ or }(0,n+1,n)\ (n\geq 0),\\
0, & \text{otherwise}.
\end{cases}
\]
\end{lem}
\begin{proof}
This follows from \cite[Theorem 3.5]{nagao-nakajima} and \cite[Lemma 3.7]{nagao-nakajima}. 
\end{proof}
\begin{thm}
The DT type invariant $\Omega^Z_{\hat{Q},w}(\mathbf{v})$ is an integer for any generic stability condition $Z$ and for any $\mathbf{v}\in\Gamma_{\hat{Q}}$.
\end{thm}
\begin{proof}
Using Lemma \ref{lem_coni}, we can check that $Z^\circ$ satisfies the conditions in \S \ref{subsec_gen}.
\end{proof}
\begin{rem}
The BPS quiver $Q_\gamma$ \textup{(}for sufficiently large $\gamma$\textup{)} is given in Figure \ref{fig_3}.
\end{rem}
\begin{figure}[htbp]
  \centering
  \input{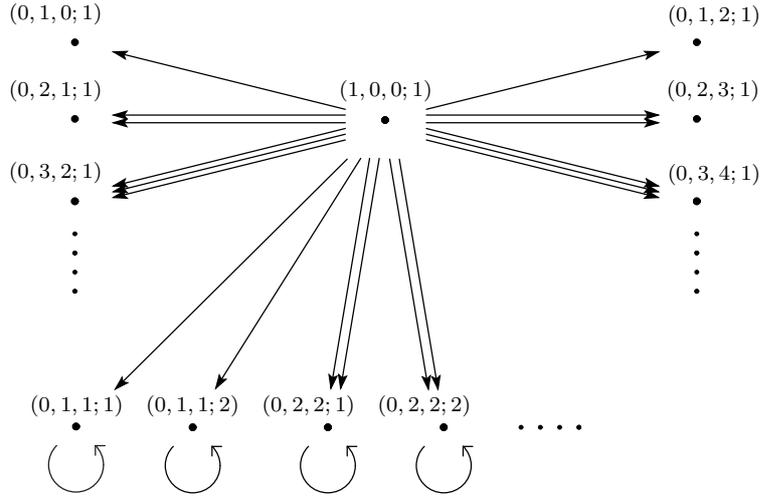}
  \caption{The BPS quiver for the conifold}
\label{fig_3}
\end{figure}

\subsubsection{}
Note that we have the special vertex $(1,0,0;1)\in I_\gamma$ in the BPS quiver $Q_\gamma$. 
We put $I^\circ_\gamma:=I_\gamma-\{(1,0,0;1)\}$ and take a subset $S\subset I^\circ_\gamma$.
Let 
\[
Z^S_{Q_\gamma}\colon \Gamma_{Q_\gamma}\to \C
\]
be a stability condition such that 
\[
\mathrm{arg}(Z^S_{Q_\gamma}(e_{(1,0,0;1)}))>\mathrm{arg}(Z^S_{Q_\gamma}(e_{s}))
\]
if $s\in S$ and 
\[
\mathrm{arg}(Z^S_{Q_\gamma}(e_{(1,0,0;1)}))<\mathrm{arg}(Z^S_{Q_\gamma}(e_{s}))
\]
if $s\in I^\circ_\gamma\backslash S$.
If $V_0\neq \emptyset$ and $V_s\neq \emptyset$ for some $s\in I^\circ_\gamma\backslash S$, then $V$ is not $Z_{Q_\gamma}^S$-semistable.

We construct the new quiver $Q_S$ by removing the vertices in $I^\circ_\gamma\backslash S$ and let $Z^+_{Q_S}\colon \Gamma_{Q_S}\to \C$ be the stability condition given by composing $Z^S_{Q_\gamma}$ and the natural injection $\iota\colon \Gamma_{Q_S}\to \Gamma_{Q_\gamma}$.
For any $\lambda\in \Gamma_{Q_S}$, we have 
\[
\Omega^{Z^+_{Q_S}}_{Q_S}(\lambda)=\Omega^{Z^S_{Q_\gamma}}_{Q_\gamma}(\iota(\lambda)).
\]
Applying the reflection functor of the quiver $Q_S$, we get the following symmetry:
\begin{thm}\label{thm_symm_coni}
For any generic stability condition $Z$, we have
\[
\Omega^Z_{\hat{Q},w}(v_\infty,v_0,v_1)=
\Omega^Z_{\hat{Q},w}(v_0-v_\infty,v_0,v_1).
\]
\end{thm}
\begin{rem}
We can apply an analogue of the reflection functor directly for the quiver with the potential $(\hat{Q},w)$.
Using the isomorphism
\[
(\C Q/I_w)^{\mr{op}}\simeq \C Q/I_w,
\]
we get another proof of Theorem \ref{thm_symm_coni}.
This argument is quite similar to Toda's one (\cite[\S 1.5]{stoppa_D0D6}).
\end{rem}
\begin{NB}
We put 
\[
\tilde{V}_S:=
\bigoplus_{
\begin{subarray}{c}
(0,v_0,v_1;j)\in I^\circ_\gamma\backslash S,\\
1\leq p\leq v_0.
\end{subarray}
}
V_{\mr{v},j}
\]
and 
\[
f_{V,S}:=\bigoplus a_{v_0,v_1,j,p}\colon V_0\to \tilde{V}_S
\]
where $a_{v_0,v_1,j,p}$ ($1\leq p\leq v_0$) is an arrow in $Q_\gamma$ from the vertex $0$ to the vertex $(0,v_0,v_1;j)$.
\begin{lem}
\end{lem}
\end{NB}

\subsubsection{}
Let ${Q}={Q}_{\mr{conifold}}$ be the quiver in Figure \ref{pic1} and $\omega$ be the following superpotential: 
\[
\omega=a_1b_1a_2b_2-a_1b_2a_2b_1.
\]
\begin{figure}[htbp]
  \centering
  \input{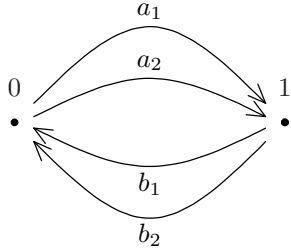}
  \caption{Non-commutative crepant resolution of the conifold}
\label{pic1}
\end{figure}
Let $P_0$ be the projective $\C Q/I_w$-module.
Note that we can canonically identify the following data:
\begin{itemize}
\item a finite dimensional $\C\hat{Q}/\hat{I}_w$-module, and
\item a triple $(W,V,s)$ of a finite dimensional vector space $W$, a finite dimensional $\C Q/I_w$-module $V$ and a homomorphism $s\colon P_0\otimes W\to V$.
\end{itemize}

Let $X=\mr{Tot}(\OO_{\mathbb{P}^1}(-1)\oplus\OO_{\mathbb{P}^1}(-1))$ be the crepant resolution of the conifold and $\pi\colon X\to \mathbb{P}^1$ be the projection.
The derived functor
\[
\R\Hom(\OO_X\oplus \pi^*\OO_{\mathbb{P}^1}(1),-)
\] 
gives an equivalence between the derived category of coherent sheaves on $X$ and the derived category of $\C/I_w$-modules, which maps the structure sheaf $\OO_X$ to the projective module $P_0$. 
We identify these two category. 

\begin{defn}
A {\it coherent system} on $X$ is a triple $(W,F,s)$ of a finite dimensional vector space $W$, a compactly supported coherent sheaf $F$ and a morphism $s\colon \OO_X \otimes W \to F$.
A coherent system $(W,F,s)$ is said to be stable if the following condition is satisfied:\begin{quote}
for any subspace $W'\subset W$ and subsheaf $F'\subset F$ such that $s(W'\otimes \OO_X)\subset F'$, we have
\begin{itemize}
\item $r(F)\cdot \dim V'< r(F')\cdot \dim V$, or
\item $r(F)\cdot \dim V'= r(F')\cdot \dim V$ and $\chi(F)\cdot \dim V'= \chi(F')\cdot \dim V$
\end{itemize}
where $r(F):=\chi(F\otimes \pi^*\OO_{\mathbb{P}^1}(1))-\chi(F)$.
\end{quote}
\end{defn}
\begin{rem}
A coherent system $(W,F,s)$ is stable, then we can check that $s$ is surjective.
\end{rem}
Given a numerical data $\gamma\in\Gamma_Q$, we take a sufficiently small $\varepsilon >0$ and let $Z_{X}=Z_{X,\gamma}\colon \Gamma_{\hat{Q}}\to C$ be a stability condition such that
\begin{align*}
\mr{arg}(Z_X(0,0,1))&>
\mr{arg}(Z_X(1,0,0))=\pi/2+\varepsilon\\
&>\mr{arg}(Z_X(0,1,1))=\pi/2>
\mr{arg}(Z_X(0,1,0)).
\end{align*}
\begin{prop}\label{prop_26}
Given a stable coherent system $(W,F,s)$ with $[F]=\gamma$, then $F\in \mr{mod}(Q,w)$ and the triple is $Z_X$-stable as a $\C\hat{Q}/\hat{I}_w$-module.
On the other hand, given a $Z_X$-stable $\C\hat{Q}/\hat{I}_w$-module $(W,V,s)$ with $[V]=\gamma$, then $V\in \mr{Coh}X$ and the triple is stable as a coherent system.
\end{prop}
\begin{proof}
We can check the claim in the same way as in \cite[Proposition 2.10]{nagao-nakajima}. 
\end{proof}
\begin{rem}
The invariant which counts objects as in Proposition \ref{prop_26} can be described as a sum of counting invariants of the quiver in Figure \ref{fig_4}
\begin{figure}[htbp]
  \centering
  \input{quiver4.tpc}
  \caption{}
\label{fig_4}
\end{figure}
\end{rem}

\subsection{BPS quivers for small crepant resolutions}
The argument as in \S \ref{subsec_BPS_conifold} can be applied in the following examples as well. 
We can construct the BPS quivers using Lemma \ref{lem_small}, \ref{lem_mckay} and \ref{lem_(0,-2)}.

\subsubsection{Toric small crepant resolutions}
Let $L_\pm$ be non-negative integers.
We put $L:=L_++L_-$ and
\[
I:=\Z/L\Z=\{0,\ldots L-1\},\quad \tilde{I}:=\Z_\mr{h}/L\Z=\{1/2,\ldots,L-1/2\}
\]
where $\Z_{\mr{h}}$ is the set of half integers. 
We take a map
\[
\sigma \colon \tilde{I}\to \{\pm\}
\]
such that $L_\pm=|\sigma^{-1}(\pm)|$.

In \cite[\S 1.2]{3tcy}, we constructed a quiver with a potential $A_\sigma=(Q_\sigma,w_\sigma)$, which is a non-commutative crepant resolution of the affine Calabi-Yau $3$-fold 
\[
\{(X,Y,Z,W)\in\C^4\mid XY=Z^{L_+}W^{L_-}\}.
\] 
The set of vertices of the quiver $Q_\sigma$ is ${I}:=\Z/L\Z$.
Put
\[
{I}_\sigma:=\{k\in{I}\mid \sigma(k-1/2)=\sigma(k+1/2)\},
\]
then $Q_\sigma$ is given by adding a loop for each vertex in ${I}_\sigma$
to the double quiver of affine $A_{L-1}$ type . 
\begin{figure}[htbp]
  \centering
  \input{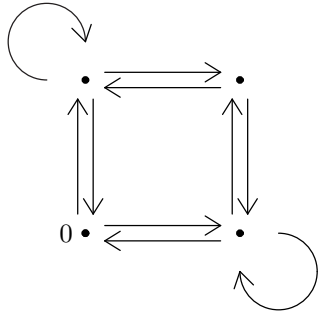}
  \caption{Quiver $Q_\sigma$ for $\sigma\colon 1/2,3/2,5/2,7/2\mapsto +,+,-,-$.}
\label{fig_6}
\end{figure}

We identify $\Z^{{I}}$ with the root lattice of affine Lie algebra of type $A_{L-1}$ and denote the set of positive root vectors (resp. positive real root vectors) by $\Lambda_+$ (resp. $\Lambda_+^{\mr{re}}$).
Let $\delta:=(1,\ldots,1)$ be the minimal imaginary root.
For $\alpha\in \Z^{{I}}$, we put 
\[
\varepsilon(\alpha):=\sum_{k\notin I_\sigma}\alpha_k.
\]

Let $\hat{Q}_\sigma$ be the quiver given from the quiver $Q_\sigma$ by adding a vertex $\infty$ and an arrow from the vertex $\infty$ to the vertex $0$.
Note that $w_\sigma$ gives a potential for $\hat{Q}_\sigma$ as well.
The following lemma is a consequence of the results in \cite{3tcy}:
\begin{lem}\label{lem_small}
\[
\Omega_{\hat{Q}_\sigma,w_\sigma}^-(r,\alpha)=
\begin{cases}
1, & r=1,\ \alpha=\vec{0},\\
-L, & r=0,\ \alpha=n\cdot\delta\, (n>0),\\
(-1)^{\varepsilon(\alpha)}, & r=0,\ \alpha\in \Lambda_+^{\mr{re}},\\
0, & \text{otherwise}.
\end{cases}
\]
\end{lem}

\subsubsection{McKay quivers for $G\subset \mr{SL}_2\subset \mr{SL}_3$}
Take a finite subgroup $G\subset\mr{SL}_2\subset \mr{SL}_3$. 
Let $Q_G$ be the McKay quiver for $G\subset \mr{SL}_3$, which is given from the one for $G\subset \mr{SL}_2$ by adding a loop for each vertex.
\begin{figure}[htbp]
  \centering
  \input{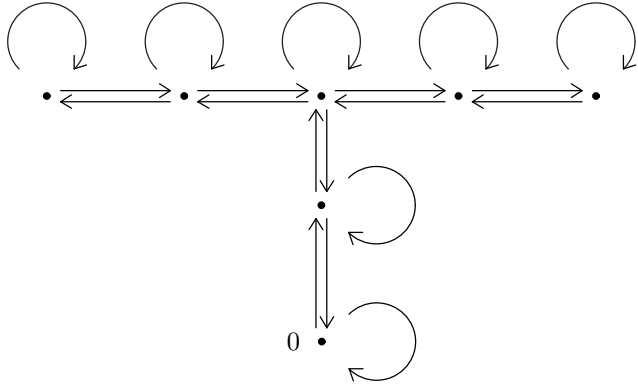}
  \caption{McKay quiver of type $E_6$}
\label{fig_6}
\end{figure}
Let $w_G$ be a potential which gives the derived McKay equivalence (\cite{ginzburg-cy}) and $0$ be the vertex which corresponds to the trivial representation.
Let $\Lambda_+^{\mr{re}}$ denote the set of positive real root and $\delta$ be the minimal imaginary root.

Let $I_G$ be the set of vertices of the quiver $Q_G$, and $\hat{Q}_G$ be the quiver given from the quiver $Q$ by adding a vertex $\infty$ and an arrow from the vertex $\infty$ to the vertex $0_G$.
See \cite{gholampour_jiang} for the following lemma:
\begin{lem}\label{lem_mckay}
\[
\Omega_{\hat{Q}_G,w_G}^-(r,\alpha)=
\begin{cases}
1, & r=1,\ \alpha=\vec{0},\\
-|I_G|, & r=0,\ \alpha=n\cdot \delta\, (n>0),\\
-1, & r=0,\ \alpha\in \Lambda_+^{\mr{re}},\\
0, & \text{otherwise}.
\end{cases}
\]
\end{lem}

\subsubsection{Obstructed $(0,-2)$-curve}
For an integer $N>1$, let ${Q}_{(0,-2);N}$ be the quiver given in Figure \ref{fig_5} and $w_{(0,-2);N}$ be the potentials given as follows:
\[
w_{(0,-2);N}=
\frac{c_0^{N+1}}{N+1}+
\frac{c_1^{N+1}}{N+1}+
c_0(b_1a_1+b_2a_2)+
c_1(a_1b_1+a_2b_2).
\]
\begin{figure}[htbp]
  \centering
  \input{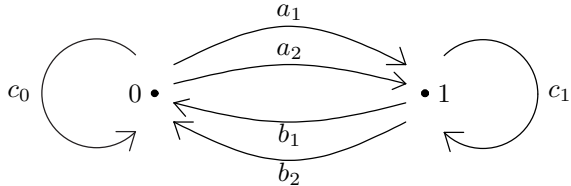}
  \caption{Non-commutative crepant resolution for obstructed $(0,-2)$-curve}
\label{fig_5}
\end{figure}

The the quiver with the potential $({Q}_{(0,-2);N},w_{(0,-2);N})$,  introduced in \cite{aspinwall_katz}, gives a non-commutative crepant resolution of the affine Calabi-Yau $3$-fold
\[
\{(X,Y,Z,U)\in \C^4\mid X^2+YZ-U^{2N}=0\}.
\]
The affine Calabi-Yau $3$-fold has an isolated singularity at the origin and the exceptional fiber of the crepant resolution of it is a $(0,-2)$ rational curve.

Let $\hat{Q}_{(0,-2);N}$ be the quiver given from the quiver ${Q}_{(0,-2);N}$ by adding a vertex $\infty$ and an arrow from the vertex $\infty$ to the vertex $0$.

We can classify stable modules in the same way as in \cite[\S 3.2]{nagao-nakajima}.
For $\alpha=(m,m+1)$ or $(m+1,m)$ ($m>0$), there exists a unique stable module, which gives a line bundle on the exceptional rational curve under the derived equivalence. 
Hence the category of semistable modules with dimension vector $\{n\cdot \alpha\mid n>0\}$ is equivalent to the category of semistable modules with dimension vector $\{(n,0)\mid n>0\}$.
The latter is equivalent to the category of $\C[u]/(u^{N})$-modules.
Since $\mr{Spec}(\C[u]/(u^{N}))$ is deformation equivalent to isolated $N$-points, we have the following (see \cite{chuang_pan} for arguments in the physics context):
\begin{NB}
Hence the moduli of stable modules with dimension vector $\alpha$ is 
\[
\mr{Spec}(\C[u]/(u^{N})),
\]
which is deformation equivalent to isolated $N$-points.
Then we can verify the following (see \cite{chuang_pan} for arguments in the physics context):
\end{NB}
\begin{lem}\label{lem_(0,-2)}
\[
\Omega_{\hat{Q}_{(0,-2);N},w_{(0,-2);N}}^-(r,\alpha)=
\begin{cases}
1, & r=1,\ \alpha=(0,0),\\
-2, & r=0,\ \alpha=(n,n)\, (n>0),\\
N, & r=0,\ \alpha=(n,n+1)\text{ or }(n+1,n)\, (n\geq 0),\\
0, & \text{otherwise}.
\end{cases}
\]
\end{lem}



\bibliographystyle{amsalpha}
\bibliography{bib-ver5}

\end{document}